\documentclass[a4paper,12pt,intlimits,oneside]{amsart}
\usepackage{amsmath}
\usepackage{amssymb}
\usepackage{amsfonts}
\usepackage{amsthm,amscd}
\usepackage{hyperref}

\numberwithin{equation}{section}
\newtheorem{theorem}{Theorem}[section]

\newtheorem{corollary}[theorem]{Corollary}

\newtheorem{proposition}[theorem]{Proposition}
\newtheorem{remark}[theorem]{Remark}

\newcommand{\comment}[1]{}

\numberwithin{equation}{section}

\def\lsim{\raisebox{-1ex}{$~\stackrel{\textstyle <}{\sim}~$}}

\theoremstyle{definition}

\begin{document}
\title []{Fractional integral and fractional maximal operators on generalized Fofana spaces}

\author[P. Nagacy]{Pokou Nagacy}
\address{Laboratoire des Sciences et Technologies de l'Environnement, UFR Environnement, Universit\'e Jean Lorougnon
GUEDE, BP 150 Daloa, C\^ote d'Ivoire}
\email{{pokounagacy@yahoo.com}}

\author[B. A. Kpata]{B\'erenger Akon Kpata}
\address{Laboratoire de Math\'ematiques et Informatique, UFR Sciences Fondamentales et
Appliquées,
 Université Nangui Abrogoua,
 02 BP 801 Abidjan 02,  C\^ote d'Ivoire}
\email{{kpata\_akon@yahoo.fr}}

\author[N. Diarra]{Nouffou Diarra}
\address{Laboratoire de Math\'ematiques et Applications, UFR Math\'ematiques et Informatique,
 Universit\'e F\'elix
 Houphou\"et-Boigny,
 22 BP 582 Abidjan 22,  C\^ote d'Ivoire}
\email{{nouffoud@yahoo.fr}}

 \renewcommand{\thefootnote}{}
\footnote{\emph{Corresponding author}: Nouffou Diarra}

\footnote{\emph{Email address:} nouffoud@yahoo.fr}



\renewcommand{\thefootnote}{\arabic{footnote}}
\setcounter{footnote}{0}

\begin{abstract}
Generalized Fofana spaces were recently introduced as generalizations of Fofana spaces and Nakai's generalized Morrey spaces. In this paper, we establish the boundedness properties of the following operators in these spaces: fractional integral operators, fractional maximal operators and generalized fractional integral operators. As a consequence, we obtain generalized Olsen-type inequalities involving the Riesz potential and generalized fractional integral operators.
\end{abstract}

\maketitle

\section{Introduction}
Let $1\leq q \leq p\leq \infty$. We denote by  $\mathcal{G}_{q,p}$ the set of all functions  $\phi :(0,\,\infty) \rightarrow(0,\infty)$ such that  $t\mapsto \phi(t) \,t^{\frac{d}{p}}$ is almost decreasing ($0 <r\leq s<\infty \Rightarrow \exists\, C>0:\phi(s) \,s^{\frac{d}{p}}\leq C \phi(r) \,r^{\frac{d}{p}}$) and $t\mapsto \phi(t)\,t^{\frac{d}{q}}$ is almost increasing ($0<r\leq s<\infty \Rightarrow \exists\, C>0:\phi(r) \,r^{\frac{d}{q}}\leq C \phi(s) \,s^{\frac{d}{q}}$).\\
For $\phi \in \mathcal{G}_{q,p}$, the generalized Fofana space $ (L^{q}, L^{p})^{\phi} (\mathbb{R}^{d})$ is defined as the set of all  functions $f\in L^{0}(\mathbb{R}^{d})$  such that  
  \begin{equation*}
  \left\|f\right\|_{(L^{q},L^{p})^{\phi}(\mathbb{R}^{d})}=\sup_{r>0} \frac{1}{\phi(r)} r^{d(-\frac{1}{q}-\frac{1}{p})}
  \: _{r}\left\|f\right\|_{q,p}<\infty,
  \end{equation*}
where $ L^{0}(\mathbb{R}^{d}) $ stands for the complex vector space of equivalent classes (modulo equality Lebesgue almost everywhere) of Lebesgue measurable complex-valued functions on $ \mathbb{R}^{d}$ and
\begin{equation*}
 _{r}\left\|f\right\|_{q,p} = \left\|\left[\int_{\mathbb R^{d}}\vert f\chi_{B(y,r)}\vert^{q}(x)dx\right]^{\frac{1}{q}}\right\|_p , \label{M}
 \end{equation*} 
with the $ L^{p}(\mathbb{R}^{d})$-norm taken with respect to the variable $ y$. Here,  
 $B(y,\, r)$ denotes the open ball centered at $y \in \mathbb{R}^{d}$ with radius $r>0$. Generalized Fofana spaces $\left(L^{q},L^{p}\right)^{\phi}(\mathbb{R}^d)$ were introduced in \cite{NKD}, in which it is proved that $\left(\left(L^{q},L^{p}\right)^{\phi}(\mathbb{R}^d),\, \left\|\cdot\right\|_{(L^{q},L^{p})^{\phi}(\mathbb{R}^{d})}  \right)$ is a Banach space. Note that, if $q \leq \alpha \leq p$  and $\phi(r)=r^{-\frac{d}{\alpha}}$  then $(L^{q}, L^{p})^{\phi} (\mathbb{R}^{d})=(L^{q}, L^{p})^{\alpha}(\mathbb{R}^{d})$, where  $(L^{q}, L^{p})^{\alpha}(\mathbb{R}^{d})$ is the classical Fofana space. Moreover $(L^{q}, L^{p})^{\phi} (\mathbb{R}^{d})$ coincides with the generalized Morrey space $\mathcal{M}_{q}^{\phi}$, defined by Nakai in \cite{EN}, when $p=\infty$.

 In this paper, we are interested in the boundedness properties of the following operators in generalized Fofana spaces: Riesz potentials, fractional maximal operators and generalized fractional integral operators.\\  
Recall that for a given function $\rho : \: (0,\, \infty) \:  \longrightarrow \: (0,\,\infty)$, the generalized fractional integral operator $ T_{\rho} $ is defined by
\begin{eqnarray*} \label{rel 9}
 T_{\rho} f(x)=\displaystyle\int_{\mathbb{R}^{d}}\frac{\rho(\left|x-y\right|)}{\left|x-y\right|^{d}}f(y)dy,\quad   x \in \mathbb R^ {d},
\end{eqnarray*}
for any suitable function $ f $ on $ \mathbb{R}^{d} $.\\
Let $0 < \gamma < d$. For $\rho = t^{\gamma}$, $T_{\rho}$ is the classical fractional integral operator $I_{\gamma}$, also known as the Riesz potential of order $\gamma$. It is closely related to the fractional maximal operator $M_{\gamma}$, defined by
\begin{eqnarray*} \label{rel 6}
 M_{\gamma} f(x)=\sup_{r>0}\vert B(x,r)\vert^{\frac{\gamma}{d}-1}\displaystyle\int_{B(x,r)}\left|f(y)\right|dy,\;\;\;\;\; x\in \mathbb{R}^{d},
\end{eqnarray*}
for any locally integrable function $f$ on $\mathbb{R}^{d}$. It follows from the definitions of $I_{\gamma}$ and $M_{\gamma}$ that
\begin{equation} \label{rel 11}
M_{\gamma} f(x) \leq C_{d}I_{\gamma} (\vert f \vert)(x),\;\;\;\; x\in \mathbb{R}^{d},
\end{equation}
where $C_{d}$ is the Lebesgue measure of the unit ball in $\mathbb{R}^{d}$.
Note that in the limiting case $\gamma=0$, the fractional maximal operator reduces to the Hardy-Littlewood maximal operator $ M$.\\
It is proved in \cite{JF} that $ M $ is bounded on classical Fofana spaces. It is also bounded on generalized Morrey spaces (see \cite{EN}) and on generalized Fofana spaces (see \cite{NKD}). The Riesz potential $I_{\gamma}$ is an important tool in harmonic analysis and partial differential equations. It acts as the inverse of the fractional Laplacian $(-\Delta)^{-\gamma/2}$. Consequently it is useful for finding explicit solutions to wave and Schr\"odinger equations. It is also found applications in the theory of Sobolev embeddings (see, for example \cite{Ma}). The boundedness property of the classical fractional integral operator and its corresponding fractional maximal operator on classical Fofana spaces have been investigated in \cite{JF, fof, Fof}. In \cite{Fof-F-K}, norm inequalities for the Riesz potential were examined in particular subspaces of Morrey spaces. For the study of the boundedness of fractional integral operators on generalized Morrey spaces the reader is referred to \cite{GE, EN, ST}. Recall that the boundedness of the operator $T_{\rho}$ on the generalized Morrey spaces $\mathcal{M}_{q}^{\phi}$ was first studied by Nakai \cite{EN1}. Several papers were devoted to norm inequalities for $T_{\rho}$ in the framework of generalized Morrey spaces, see for example \cite{GE, S-S-T, Su} and the references therein.

In studying a Schr\"odinger equation with perturbed potentials $ W  $ (functions on $ \mathbb{R}^{d} $), Olsen \cite{Ol} established
the boundedness of the multiplication operator $ f \mapsto W \cdot I_{\gamma}f$ on Morrey spaces. This estimate, later called Olsen's inequality, was improved by Sawano et al. in \cite{SST}, where they showed its applications to partial differential equations and the potential theory.  Olsen's inequality was also extended not only to generalized Morrey spaces but also to the operator $T_{\rho}$, see \cite{GE} and \cite{KNS}. In this note, we shall prove Olsen-type inequalities in generalized Fofana spaces involving the Riesz potential and generalized fractional integral operators.

This paper is organized as follows. In Section \ref{S2}, we give some properties of generalized Fofana spaces and we recall the boundedness result of the Hardy-Littlewwod maximal operator on these spaces. Section \ref{S3} is devoted to norm inequalities for the classical fractional integral operator and its corresponding fractional maximal operator in generalized Fofana spaces. In Section \ref{S4}, we investigate the boundedness of generalized fractional integral operators in generalized Fofana spaces. Finally, in Section \ref{S5}, we establish Olsen-type inequalities. \\ 
Throughout this note, the letter $C$ will be used as a generic positive constant not depending on the relevant variables. Its value may change from one occurrence to another. We shall also use the abbreviation $ A \lsim \mathrm{ B}$ for the inequalities $ A\leq C B$. If $ A \lsim B$ and $ B \lsim  A$, then we write $ A\approx  B$. We adopt the usual convention $\frac{1}{\infty} =0$.

\section{Preliminaries}
\label{S2}
 Let $1\leq q,\, p \leq\infty$ and let $r>0$. Introduced by Wiener \cite{W} and systematically studied by Holland \cite{FH}, the amalgam spaces $(L^{q},\, l^{p})_{r}(\mathbb{R}^{d})$ are defined as the set of all functions $f \in L^{0}(\mathbb{R}^{d})$ satisfying $ \: _{r}\widetilde{\left\|f\right\|}_{q,p}<\infty$, where 
 $$_{r}\widetilde{\left\|f\right\|}_{q,p}= \left\|\left\{  \left\|f\chi_{Q^{r}_{k}}\right\|_{q} \right\}_{k \in \mathbb{Z}^{d} }\right\|_{\ell^{p}},  $$
with $k=(k_{1},\, ...,\,k_{d})\in \mathbb{Z}^{d}$ and
 $ Q^{r}_{k} = \prod_{i=1}^{d}  [rk_{i}, r(k_{i}+1)) $.\\
For more informations about amalgam spaces, we refer to \cite{FS2} and the references therein. We recall the following properties. 
 \begin{proposition} 
 \label{prop 3} (See \cite{FH} or  \cite{FS2}.)
Let $r>0$.
\begin{enumerate} 
\item Let $1\leq q,\, p \leq\infty$. Then, the following assertions
hold.
\begin{enumerate}
\item  $((L^{q},\, l^{p})_{r}(\mathbb{R}^{d}),   \: _{r}\widetilde{\left\|\cdot  \right\|}_{q,p}) $ is a complex Banach space.
\item  There exists a constant $ C>0 $ such that
$$C^{-1}\:r^{\frac{d}{p}}\:_{r}\widetilde{\left\|f\right\|}_{q,p}\leq \: _{r}\left\|f\right\|_{q,p} \leq C r^{\frac{d}{p}}\:_{r}\widetilde{\left\|f\right\|}_{q,p},\;\;\;\;\; f \in L^{0}(\mathbb{R}^{d}).$$
\end{enumerate}
\item Let $1\leq q_1,\, p_1 \leq \infty$ and Let $1\leq q_2,\, p_2 \leq \infty$ such that
$$\frac{1}{q_1}+ \frac{1}{q_2}=\frac{1}{q} \leq 1 \text{ and }  \frac{1}{p_1}+ \frac{1}{p_2}=\frac{1}{p} \leq 1. $$
If $f \in (L^{q_1},\, l^{p_1})_{r}(\mathbb{R}^{d})$ and $g \in (L^{q_2},\, l^{p_2})_{r}(\mathbb{R}^{d})$, then we have $fg \in (L^{q},\, l^{p})_{r}(\mathbb{R}^{d})$. Moreover
$$ _{r}\widetilde{\left\|fg\right\|}_{q,p}\leq \,  _{r}\widetilde{\left\|f\right\|}_{q_1,p_1}\, _{r}\widetilde{\left\|g\right\|}_{q_2,p_2}.$$
\end{enumerate}
 \end{proposition} 
Let $1\leq q \leq p\leq \infty$. The fact that the function $\phi$ belongs to the set $\mathcal{G}_{q,p}$ ensures that the generalized Fofana space  is non-trivial. Furthermore, if $\phi \in \mathcal{G}_{q,p}$ then there exist constants $C_1>0$ and $C_2>0$ such that:
 \begin{equation}\label{eq 1}
 \frac{1}{2} \leq \frac{r}{s} \leq 2 \implies \frac{1}{C_1} \leq \frac{\phi (r)}{\phi(s)} \leq C_1;
 \end{equation}
 and  
\begin{equation}
\label{lem 2}
  (2^{i}r)^{d(\frac{1}{q}+\frac{1}{p})} \phi(2^{i+1}r) \leq C_2 \displaystyle\int_{2^{i}r}^{2^{i+1}r} \phi(t)t^{d(\frac{1}{q}+\frac{1}{p})-1} dt, 
 \end{equation}  
for any positive integer $i$ and for any $r>0$.\\
Note that the doubling condition \eqref{eq 1} is well-known and the proof of inequality \eqref{lem 2} is similar to that of Lemma 3.3 in \cite{NKD}.

The family of spaces $ (L^{q},L^{p})^{\phi}(\mathbb{R}^{d}) $ is increasing (in the sense of inclusion) with respect to the $ p $ power. More precisely, we have the following proposition. 
 \begin{proposition}(\cite[Proposition 2.7 (2)]{NKD})
 \label{prop 1}
 Let  $ f \in L^{0}(\mathbb{R}^{d}). $ If $1\leq q\leq p_{1} \leq p_{2} \leq\infty$ and $ \phi \in \mathcal{G}_{q,p_{1}} $, then 
 $$\left\| f \right\|_{(L^{q},L^{p_{2}})^{\phi}(\mathbb{R}^{d})} \lesssim \left\|f\right\|_{(L^{q},L^{p_{1}})^{\phi}(\mathbb{R}^{d})},$$ and consequently $(L^{q},L^{p_{1}})^{\phi}(\mathbb{R}^{d})\subset (L^{q},L^{p_{2}})^{\phi}(\mathbb{R}^{d}).$
 \end{proposition}
We now state the following H\"older-type inequality in generalized Fofana spaces.
\begin{proposition}\label{prop 2}
Assume $1\leq q_{1} \leq p_{1}\leq \infty$ and $1\leq q_{2} \leq p_{2}\leq \infty$.
Let $\phi \in \mathcal{G}_{q,p} , \phi_{1} \in \mathcal{G}_{q_{1},p_{1}} $ and $\phi_{2} \in \mathcal{G}_{q_{2},p_{2}} $ such that $\phi_{1}(r) \phi_{2}(r) \leq \phi(r) $ for all $r>0$. If
	$$\frac{1}{q_{1}} + \frac{1}{q_{2}} = \frac{1}{q} \leq 1  \text{ and } \frac{1}{p_{1}} + \frac{1}{p_{2}} = \frac{1}{p} \leq 1,$$
then for all $(f,\,g) \in L^{0} (\mathbb{R}^{d})\times L^{0} (\mathbb{R}^{d})$ we have
	$$\Vert f g\Vert_{(L^{q},L^{p})^{\phi}(\mathbb{R}^{d})} \lesssim \Vert f \Vert_{(L^{q_{1}},L^{p_{1}})^{\phi_{1}}(\mathbb{R}^{d}) } \Vert g \Vert_{(L^{q_{2}},L^{p_{2}})^{\phi_{2}}(\mathbb{R}^{d})}.$$
\end{proposition}

\begin{proof}
Let $(f,\,g) \in L^{0} (\mathbb{R}^{d})\times L^{0} (\mathbb{R}^{d})$. Suppose that $\Vert f g\Vert_{(L^{q},L^{p})^{\phi}(\mathbb{R}^{d})} \neq 0$, $\Vert f \Vert_{(L^{q_{1}},L^{p_{1}})^{\phi_{1}}(\mathbb{R}^{d}) } < \infty$ and  $ \Vert g \Vert_{(L^{q_{2}},L^{p_{2}})^{\phi_{2}}(\mathbb{R}^{d})}< \infty$, since otherwise the desired result is obvious.\\ Since $_{r}\left\|f\right\|_{q_{1},p_{1}}<\infty$ and $_{r}\left\|f\right\|_{q_{2},p_{2}}< \infty$ for all $r>0$, by applying Proposition \ref{prop 3}, we obtain 
\begin{equation*}
    _{r}\left\|fg\right\|_{q,p} \lesssim \: _{r}\left\|f\right\|_{q_{1},p_{1}} \: _{r}\left\|f\right\|_{q_{2},p_{2}}
\end{equation*}
for all $r>0$. Multiplying both sides of the previous inequality by
$\frac{1}{\phi(r)}r^{d(-\frac{1}{q}-\frac{1}{p})}$, we get 
\begin{eqnarray*}
&&\frac{1}{\phi(r)}r^{d(-\frac{1}{q}-\frac{1}{p})}\: _{r}\left\|fg\right\|_{q,p} \\ & \lesssim &  \frac{1}{\phi_{1}(r)}r^{d(-\frac{1}{q_{1}}-\frac{1}{p_{1}})}\:_{r}\left\|f\right\|_{q_{1},p_{1}} \frac{1}{\phi_{2}(r)}r^{d(-\frac{1}{q_{2}}-\frac{1}{p_{2}})}\: _{r}\left\|f\right\|_{q_{2},p_{2}}\\
& \lesssim & \Vert f \Vert_{(L^{q_{1}},L^{p_{1}})^{\phi_{1}}(\mathbb{R}^{d}) } \Vert g \Vert_{(L^{q_{2}},L^{p_{2}})^{\phi_{2}}(\mathbb{R}^{d})}.
\end{eqnarray*}
We conclude the proof by taking in the left-hand side, the supremum  over all $r>0$.
\end{proof}
We end this section by recalling the boundedness property of the Hardy-Littlewood maximal operator on generalized Fofana spaces.
\begin{theorem}(\cite[Theorem 3.1]{NKD})
\label{theo 1}
 Let $1 < q\leq p \leq \infty$ and $\phi\in \mathcal{G}_{q,p}$. Assume that there is a constant $ C > 0 $ such that for any $ r >0 $  
  \begin{equation}\label{eqq 2}
 \displaystyle\int_{r}^{\infty} \phi^{q}(t)t^{\frac{dq}{p}-1} dt \leq C   \phi^{q}(r)r^{\frac{dq}{p}}.
  \end{equation}
  Then $$\left\Vert Mf \right\Vert_{(L^{q},L^{p})^{\phi}(\mathbb{R}^{d})} \lsim  \left\Vert f  \right\Vert_{(L^{q},L^{p})^{\phi}(\mathbb{R}^{d})}, \quad f \in (L^{q},L^{p})^{\phi}(\mathbb{R}^{d}).$$
\end{theorem}
\section{Norm inequalities for classical fractional integral and fractional maximal operators}
\label{S3}
Let us begin this section with the following notation used later in the work. For $1<q<\infty$, we shall denote by $q'$ its conjugate exponent: $\frac{1}{q}+\frac{1}{q'}=1$. Our first result is the below theorem stating the boundedness property of the classical fractional integral operator in generalized Fofana spaces.  
\begin{theorem}\label{theo 2}
Assume that $1 < q\leq p \leq \infty$, $ 0 < \gamma < \frac{d}{q}$ and $-\frac{d}{q} \leq \beta  < -\gamma $. Let $\phi\in \mathcal{G}_{q,p}$ satisfying condition (\ref{eqq 2}) and $\phi(r) \lsim  r^{\beta}$ for every $r > 0$. \\ Put $$\frac{1}{\bar{p}} = \frac{1}{p}\left(1 +\frac{\gamma}{\beta}\right) \quad \text{ and } \quad \frac{1}{\bar{q}} = \frac{1}{q}\left(1 +\frac{\gamma}{\beta}\right). $$
 Then, for all  $f \in (L^{q},L^{p})^{\phi}(\mathbb{R}^{d})$, we have
 $$\left\Vert I_{\gamma}f \right\Vert_{(L^{\bar{q}},L^{\bar{p}})^{\psi}(\mathbb{R}^{d})} \lsim  \left\Vert f  \right\Vert_{(L^{q},L^{p})^{\phi}(\mathbb{R}^{d})},$$ where  $\psi(r)= (\phi(r))^{\frac{q}{\bar{q}}}$.
\end{theorem}
\begin{remark}
Observe that Theorem \ref{theo 2} in the case $p=\infty$ was proved in \cite[Theorem 2.2]{GE}. Moreover,
when $\phi (r) = r^{-\frac{d}{\alpha}}$, where $1<q \leq \alpha \leq  p\leq \infty$ and $ 0< \frac{\gamma}{d} < \frac{1}{\alpha}$, Theorem \ref{theo 2} reduces to Corollary 4.4 (1) in \cite{JF}. 
\end{remark}
\begin{proof}[Proof of Theorem \ref{theo 2}]
Let $1 < q\leq p < \infty$ and let $\phi\in \mathcal{G}_{q,p}$ satisfying condition (\ref{eqq 2}). Suppose that $f \in (L^{q},L^{p})^{\phi}(\mathbb{R}^{d})\setminus \{ 0 \}$ since the case $f=0$ is obvious. Let $r>0$ and let $x \in \mathbb{R}^{d}$. We have
$$ I_{\gamma} f(x)=\displaystyle\int_{B(x,\, r)}\frac{f(y)}{\left|x-y\right|^{d-\gamma}}dy+\displaystyle\int_{\mathbb{R}^{d}\setminus B(x,\, r)}\frac{f(y)}{\left|x-y\right|^{d-\gamma}}dy = A(x) + B(x) .  $$
Following the proof of Theorem 2.2 of \cite{GE}, we may write on the one hand,
\begin{eqnarray} \label{rel 12}
& & \vert A(x) \vert= \left\vert \displaystyle\int_{B(x,\, r)}\frac{f(y)}{\left|x-y\right|^{d-\gamma}}dy \right\vert \leq C r^{\gamma} M f(x).
\end{eqnarray}
On the other hand, we have
\begin{eqnarray*}
& &\vert B(x) \vert  \\
& \leq & \sum\limits^{\infty}_{i=0} (2^{i}r)^{-d+\gamma} \vert B(x,2^{i+1}r) \vert^{\frac{d}{q'}} \left( \displaystyle\int_{B(x,2^{i+1}r)} \vert f(y)\vert^{q} dy\right)^{\frac{1}{q}}\\
 &\leq & C \sum\limits^{\infty}_{i=0} (2^{i}r)^{-d+\gamma} (2^{i+1}r)^{\frac{d}{q'}}(2^{i+1}r)^{\frac{d}{q}} \phi(2^{i+1}r) \\ & & \times \frac{1}{\phi(2^{i+1}r)}(2^{i+1}r)^{-\frac{d}{q}}\left( \displaystyle\int_{B(x,2^{i+1}r)} \vert f(y)\vert^{q} dy\right)^{\frac{1}{q}}\\
 &\leq & C \Vert f \Vert_{(L^{q},L^{\infty})^{\phi}(\mathbb{R}^{d})}\sum\limits^{\infty}_{i=0} (2^{i}r)^{-d+\gamma} (2^{i+1}r)^{d} \phi(2^{i+1}r).
\end{eqnarray*}
Since $\phi(r) \lsim  r^{\beta}$ for every $r > 0$, we get
\begin{eqnarray*} 
\vert B(x) \vert &\leq & C \Vert f \Vert_{(L^{q},L^{\infty})^{\phi}(\mathbb{R}^{d})}\sum\limits^{\infty}_{i=0} (2^{i}r)^{-d+\gamma} (2^{i+1}r)^{d} (2^{i+1}r)^{\beta}\\&\leq & C r^{\beta +\gamma}\Vert f \Vert_{(L^{q},L^{\infty})^{\phi}(\mathbb{R}^{d})}\sum\limits^{\infty}_{i=0} (2^{i})^{\beta +\gamma}.
\end{eqnarray*}
Hence,
\begin{eqnarray}\label{rel 13}
\vert B(x) \vert \leq C r^{\beta +\gamma}\Vert f \Vert_{(L^{q},L^{\infty})^{\phi}(\mathbb{R}^{d})}.
\end{eqnarray}
From \eqref{rel 12} and \eqref{rel 13}, we deduce that
\begin{eqnarray*}
\vert I_{\gamma} f(x) \vert &\leq& C\left(r^{\gamma} M f(x)+r^{\beta +\gamma}\Vert f \Vert_{(L^{q},L^{\infty})^{\phi}(\mathbb{R}^{d})}\right)
\end{eqnarray*}
for all $ r > 0 $. Thus, taking $r = \left( \frac{\Vert f \Vert_{(L^{q},L^{\infty})^{\phi}(\mathbb{R}^{d})}}{Mf(x)}         \right)^{\frac{1}{\beta}},$ we get 
\begin{eqnarray*}
\label{rel 33}
\vert I_{\gamma} f(x) \vert &\leq& C (Mf(x))^{\frac{\gamma + \beta }{\beta}} \Vert f \Vert^{-\frac{\gamma}{\beta}}_{(L^{q},L^{\infty})^{\phi}(\mathbb{R}^{d})}
\end{eqnarray*}
for all $ x \in \mathbb{R}^{d}$. It follows that 
\begin{eqnarray*}
_{r}\Vert I_{\gamma} f \Vert_{\bar{q},\bar{p}} &\leq& C\Vert f \Vert^{-\frac{\gamma}{\beta}}_{(L^{q},L^{\infty})^{\phi}(\mathbb{R}^{d})}\:_{r}\Vert M f \Vert^{1+ \frac{\gamma}{\beta}}_{\bar{q}(1+ \frac{\gamma}{\beta}),\bar{p}(1+ \frac{\gamma}{\beta})}\\
&\leq& C\Vert f \Vert^{-\frac{\gamma}{\beta}}_{(L^{q},L^{\infty})^{\phi}(\mathbb{R}^{d})}\:_{r}\Vert M f \Vert^{1+ \frac{\gamma}{\beta}}_{q,p}.
\end{eqnarray*}
Then, multiplying both sides of the previous inequality by\\ $ \frac{1}{\psi(r)} r^{d(-\frac{1}{\bar{q}}-\frac{1}{\bar{p}})} $, where  $\psi(r)= (\phi(r))^{\frac{q}{\bar{q}}}$, we get 
\begin{eqnarray*}
&&\frac{1}{\psi(r)} r^{d(-\frac{1}{\bar{q}}-\frac{1}{\bar{p}})} \:_{r}\Vert I_{\gamma} f \Vert_{\bar{q},\bar{p}} \\ &\leq& C\Vert f \Vert^{-\frac{\gamma}{\beta}}_{(L^{q},L^{\infty})^{\phi}(\mathbb{R}^{d})}    \left(  \frac{1}{\phi(r)} r^{d(-\frac{1}{q}-\frac{1}{p})}\:_{r}\Vert M f \Vert_{q,p}  \right)^{1+ \frac{\gamma}{\beta}}.
\end{eqnarray*}
Hence,  $$\left\Vert I_{\gamma}f \right\Vert_{(L^{\bar{q}},L^{\bar{p}})^{\psi}(\mathbb{R}^{d})} \leq C\Vert f \Vert^{-\frac{\gamma}{\beta}}_{(L^{q},L^{\infty})^{\phi}(\mathbb{R}^{d})} \left\Vert Mf  \right\Vert^{1+ \frac{\gamma}{\beta}}_{(L^{q},L^{p})^{\phi}(\mathbb{R}^{d})}.$$
By Theorem \ref{theo 1}, we get 
\begin{eqnarray*}
\label{rel 44}
\left\Vert I_{\gamma}f \right\Vert_{(L^{\bar{q}},L^{\bar{p}})^{\psi}(\mathbb{R}^{d})} \leq C\Vert f \Vert^{-\frac{\gamma}{\beta}}_{(L^{q},L^{\infty})^{\phi}(\mathbb{R}^{d})} \left\Vert f  \right\Vert^{1+ \frac{\gamma}{\beta}}_{(L^{q},L^{p})^{\phi}(\mathbb{R}^{d})}.
\end{eqnarray*}
We end the proof by appying Proposition \ref{prop 1}.
\end{proof}
As an immediate consequence of Theorem \ref{theo 2} and inequality (\ref{rel 11}), we have the
following result.
\begin{corollary}
Assume that $1 < q\leq p \leq \infty$, $ 0 < \gamma < \frac{d}{q}$ and $-\frac{d}{q} \leq \beta  < -\gamma $. Let $\phi\in \mathcal{G}_{q,p}$ satisfying condition (\ref{eqq 2}) and $\phi(r) \lsim  r^{\beta}$ for every $r > 0$. \\ Put $$\frac{1}{\bar{p}} = \frac{1}{p}\left(1 +\frac{\gamma}{\beta}\right) \quad \text{ and } \quad \frac{1}{\bar{q}} = \frac{1}{q}\left(1 +\frac{\gamma}{\beta}\right). $$
 Then, for all  $f \in (L^{q},L^{p})^{\phi}(\mathbb{R}^{d})$, we have
 $$\left\Vert M_{\gamma}f \right\Vert_{(L^{\bar{q}},L^{\bar{p}})^{\psi}(\mathbb{R}^{d})} \lsim  \left\Vert f  \right\Vert_{(L^{q},L^{p})^{\phi}(\mathbb{R}^{d})},$$ where  $\psi(r)= (\phi(r))^{\frac{q}{\bar{q}}}$.
\end{corollary}

The next theorem gives a norm equivalence of the Riesz potential and its corresponding fractional maximal operator when we deal with non-negative
measurable functions. 
\begin{theorem}
\label{Theo 3.4}
Let $1\leq  q \leq p \leq \infty$, $0<\gamma< d$ and $\phi \in \mathcal{G}_{q,p}$. Assume that there exists a constant $ C>0 $ such that for any $ r>0 $,
\begin{equation}\label{eq 44}
 \displaystyle\int_{r}^{\infty} \phi(t)t^{d(\frac{1}{q}+\frac{1}{p})-1} dt \leq C   \phi(r)r^{d(\frac{1}{q}+\frac{1}{p})}.
\end{equation}
 Then, for all positive functions $ f \in L_{loc}^{q}(\mathbb{R}^{d})$, we have
$$ \Vert I_{\gamma}f \Vert_{(L^{q},L^{p})^{\phi}(\mathbb{R}^{d})}  \approx   \Vert M_{\gamma}f \Vert_{(L^{q},L^{p})^{\phi}(\mathbb{R}^{d})}. $$
\end{theorem}
\begin{proof}
Suppose $1\leq  q \leq p \leq \infty$ and $0<\gamma< d$. Let $\phi \in \mathcal{G}_{q,p}$ satisfying \eqref{eq 44}. In view of inequality \eqref{rel 11} it suffices to prove that $$\Vert I_{\gamma}f \Vert_{(L^{q},L^{p})^{\phi}(\mathbb{R}^{d})}  \lesssim   \Vert M_{\gamma}f \Vert_{(L^{q},L^{p})^{\phi}(\mathbb{R}^{d})}.$$
Let $r>0$ and let $y \in \mathbb{R}^{d}$. By Lemma 4.2 of \cite{GoM}, we have
\begin{equation}\label{rel 16}
 \Vert (I_{\gamma}f) \chi_{B(y,r)} \Vert_{q} \approx \Vert (M_{\gamma}f) \chi_{B(y,r)} \Vert_{q} + \vert B(y,r) \vert^{\frac{1}{q}} \int_{\mathbb{R}^{d}\setminus B(y,r)}\frac{f(x)}{\left|x-y\right|^{d-\gamma}}dx.
\end{equation}
Furthermore,
\begin{eqnarray*}
& & \vert B(y,r) \vert^{\frac{1}{q}} \int_{\mathbb{R}^{d}\setminus B(y,r)}\frac{f(x)}{\left|x-y\right|^{d-\gamma}}dx \\
 &=&  \sum\limits^{\infty}_{i=0} \vert B(y,r) \vert^{\frac{1}{q}}  \displaystyle\int_{ B(y,\, 2^{i+1}r)\setminus B(y,\, 2^{i}r) } \frac{f(x)}{\left|x-y\right|^{d-\gamma}} dx\\ &\lesssim&  \sum\limits^{\infty}_{i=0} (2^{i})^{-\frac{d}{q}}\vert B(y,r) \vert^{\frac{\gamma}{d}-(1-\frac{1}{q})}  \Vert f \chi_{B(y,2^{i+1}r)} \Vert_{1}.
\end{eqnarray*}
Then, thanks to \cite[Theorem 5.2]{GoM}, we get
\begin{eqnarray*}
\vert B(y,r) \vert^{\frac{1}{q}} \int_{\mathbb{R}^{d}\setminus B(y,r)}\frac{f(x)}{\left|x-y\right|^{d-\gamma}}dx \lesssim  \sum\limits^{\infty}_{i=0} (2^{i})^{-\frac{d}{q}}  \Vert (M_{\gamma}f) \chi_{B(y,2^{i+1}r)} \Vert_{q}.
\end{eqnarray*}
Taking into
account this inequality, the $ L^{p}$-norm of both sides of \eqref{rel 16}, with respect to $y$, leads to
$$_{r}\Vert I_{\gamma}f \Vert_{q,p} \lesssim \: _{r}\Vert M_{\gamma}f \Vert_{q,p} +  \sum\limits^{\infty}_{i=0} (2^{i})^{-\frac{d}{q}}  \: _{2^{i+1}r}\Vert M_{\gamma}f \Vert_{q,p}.  $$
We have, on the one hand,
\begin{eqnarray} \label{eq 5}
_{r}\Vert M_{\gamma}f \Vert_{q,p} \leq \phi(r)r^{d(\frac{1}{q}+\frac{1}{p})}\Vert M_{\gamma}f \Vert_{(L^{q},L^{p})^{\phi}(\mathbb{R}^{d})}.
\end{eqnarray}
On the other hand,
\begin{eqnarray*}
& &\sum\limits^{\infty}_{i=0} (2^{i})^{-\frac{d}{q}}  \: _{2^{i+1}r}\Vert M_{\gamma}f \Vert_{q,p}\\
&\leq& \left\Vert M_{\gamma}f \right\Vert_{(L^{q},L^{p})^{\phi }(\mathbb{R}^{d})} \sum^{\infty}_{i=0} \phi(2^{i+1}r)(2^{i+1}r)^{d(\frac{1}{q}+\frac{1}{p})}.
\end{eqnarray*}
It follows from \eqref{lem 2} that
\begin{eqnarray*}
& &\sum\limits^{\infty}_{i=0} (2^{i})^{-\frac{d}{q}}  \: _{2^{i+1}r}\Vert M_{\gamma}f \Vert_{q,p}\\ 
&\lsim & \left\Vert M_{\gamma}f \right\Vert_{(L^{q},L^{p})^{\phi }(\mathbb{R}^{d})} \sum^{\infty}_{i=0} \displaystyle\int_{2^{i}r}^{2^{i+1}r} \phi(t)t^{d(\frac{1}{q}+\frac{1}{p})-1} dt. 
\end{eqnarray*}
Hence,
$$\sum\limits^{\infty}_{i=0} (2^{i})^{-\frac{d}{q}}  \: _{2^{i+1}r}\Vert M_{\gamma}f \Vert_{q,p}  \lsim  \left\Vert M_{\gamma}f \right\Vert_{(L^{q},L^{p})^{\phi }(\mathbb{R}^{d})} \displaystyle\int_{r}^{\infty} \phi(t)t^{d(\frac{1}{q}+\frac{1}{p})-1} dt.  $$
Since $\phi$ satisfies \eqref{eq 44}, we get 
\begin{equation}\label{eq 6}
\sum\limits^{\infty}_{i=0} (2^{i})^{-\frac{d}{q}}  \: _{2^{i+1}r}\Vert M_{\gamma}f \Vert_{q,p} \lsim  \left\Vert M_{\gamma}f \right\Vert_{(L^{q},L^{p})^{\phi }(\mathbb{R}^{d})}\phi(r)r^{d(\frac{1}{q}+\frac{1}{p})}.
\end{equation}
From \eqref{eq 5} and \eqref{eq 6}, we deduce that 
\begin{equation*} 
\label{eq 7} 
\frac{1}{\phi(r)}r^{d(-\frac{1}{q}-\frac{1}{p})} \: _{r}\Vert I_{\gamma}f \Vert_{q,p} \lsim \left\Vert M_{\gamma}f \right\Vert_{(L^{q},L^{p})^{\phi }(\mathbb{R}^{d})}.
\end{equation*}
We obtain the desired result by taking the supremum over all $ r > 0 $ in
the left-hand side of this last inequality.
\end{proof}
Note that Theorem \ref{Theo 3.4} generalizes Theorem 4.7 in \cite{JF} and Theorem 7.1 in \cite{GoM}.

\section{Boundedness of generalized fractional integral operators}
\label{S4}
In this section, under different appropriate conditions on $\phi$ and $\rho$, we establish two results presenting the boundedness of the generalized fractional integral operator $T_{\rho}$ from the generalized Fofana space $(L^{q},L^{p})^{\phi}(\mathbb{R}^{d})$ to another one. 
\begin{theorem}\label{theo 3}
Suppose that $1 < q\leq p \leq \infty$, $ 0 < \gamma < \frac{d}{q}$ and $-\frac{d}{q} \leq \beta  < -\gamma $. Let $\phi\in \mathcal{G}_{q,p}$ satisfying condition (\ref{eqq 2}) and let $\rho : \: (0,\,\infty) \:  \to \: (0,\,\infty) \:$ satisfying condition (\ref{eq 1}). Assume also that 
$\phi(r) \lsim  r^{\beta}$ and $\rho(r) \lsim  r^{\gamma}$ for every $r > 0$. \\ Put $$\frac{1}{\bar{p}} = \frac{1}{p}\left(1 +\frac{\gamma}{\beta}\right) \quad \text{ and } \quad \frac{1}{\bar{q}} = \frac{1}{q}\left(1 +\frac{\gamma}{\beta}\right). $$
Then, for all  $f \in (L^{q},L^{p})^{\phi}(\mathbb{R}^{d})$ we have
 $$\left\Vert T_{\rho}f \right\Vert_{(L^{\bar{q}},L^{\bar{p}})^{\psi}(\mathbb{R}^{d})} \lsim  \left\Vert f  \right\Vert_{(L^{q},L^{p})^{\phi}(\mathbb{R}^{d})},$$ where  $\psi(r)= (\phi(r))^{\frac{q}{\bar{q}}}$.
\end{theorem}
\begin{remark}
Note that Theorem \ref{theo 3} in the case $p=\infty$ was proved in \cite[Theorem 2.3]{GE}. 
\end{remark}
\begin{proof}[Proof of Theorem \ref{theo 3}]
Let $1 < q\leq p < \infty$. Let $\phi\in \mathcal{G}_{q,p}$ satisfying condition (\ref{eqq 2}) and $\rho : \: (0,\,\infty) \:  \to \: (0,\,\infty) \:$ satisfying condition (\ref{eq 1}). We suppose that $f \in (L^{q},L^{p})^{\phi}(\mathbb{R}^{d})\setminus \{ 0 \}$ since the case $f=0$ is obvious. Let $r>0$ and let $x \in \mathbb{R}^{d}$. We have
\begin{eqnarray*}
T_{\rho} f(x)&=&\displaystyle\int_{B(x,\, r)}\frac{\rho(\left|x-y\right|)}{\left|x-y\right|^{d}}f(y)dy+\displaystyle\int_{\mathbb{R}^{d}\setminus B(x,\, r)}\frac{\rho(\left|x-y\right|)}{\left|x-y\right|^{d}}f(y)dy \\ &=& I_{1}(x) + I_{2}(x) . 
\end{eqnarray*}
Following the proof of Theorem B in \cite{GE0}, we may write
\begin{eqnarray*} 
 \vert I_{1}(x) \vert & \leq &  C M f(x) \sum\limits^{-1}_{i=-\infty}  \rho(2^{i}r).
\end{eqnarray*}
It follows that $$\vert I_{1}(x) \vert  \leq  C M f(x) \sum\limits^{-1}_{i=-\infty}  (2^{i}r)^{\gamma}. $$ 
Hence, \begin{eqnarray} \label{rel 10}
\vert I_{1}(x) \vert  \leq C   M f(x)r^{\gamma}.
\end{eqnarray}
Furthermore, we have
\begin{eqnarray*}
& &\vert I_{2}(x) \vert  \\
& \leq & C\sum\limits^{\infty}_{i=0}\frac{\rho(2^{i}r)}{(2^{i}r)^{d}} \vert B(x,2^{i+1}r) \vert^{\frac{d}{q'}} \left( \displaystyle\int_{B(x,2^{i+1}r)} \vert f(y)\vert^{q} dy\right)^{\frac{1}{q}}\\
 &\leq & C \sum\limits^{\infty}_{i=0}\frac{\rho(2^{i}r)}{(2^{i}r)^{\frac{d}{q}}} (2^{i+1}r)^{\frac{d}{q}} \phi(2^{i+1}r)\frac{1}{\phi(2^{i+1}r)}(2^{i+1}r)^{-\frac{d}{q}} \Vert f \chi_{B(x,2^{i+1}r)}   \Vert_{q} \\
 &\leq & C \Vert f \Vert_{(L^{q},L^{\infty})^{\phi}(\mathbb{R}^{d})}\sum\limits^{\infty}_{i=0} \rho(2^{i}r) \phi(2^{i}r)\\&\leq & C \Vert f \Vert_{(L^{q},L^{\infty})^{\phi}(\mathbb{R}^{d})}\sum\limits^{\infty}_{i=0} (2^{i}r)^{\gamma} (2^{i}r)^{\beta}.
\end{eqnarray*}
Hence,
\begin{eqnarray}\label{rel 14}
\vert I_{2}(x) \vert \leq C r^{\beta +\gamma}\Vert f \Vert_{(L^{q},L^{\infty})^{\phi}(\mathbb{R}^{d})}.
\end{eqnarray}
It follows from \eqref{rel 10} and \eqref{rel 14} that
\begin{eqnarray*}
\vert T_{\rho} f(x) \vert &\leq& C\left(r^{\gamma} M f(x)+r^{\beta +\gamma}\Vert f \Vert_{(L^{q},L^{\infty})^{\phi}(\mathbb{R}^{d})}\right)
\end{eqnarray*}
for all $ r > 0 $. \\ 
We end the proof as we did in that of Theorem \ref{theo 2}.
\end{proof}
A further statement of the boundedness of the generalized fractional integral operator $T_{\rho}$ in generalized Fofana spaces reads as follows.
\begin{theorem}\label{theo 4}
Suppose that $1 < q\leq p < \infty$. Let $\phi$ be a surjective function belonging to $ \mathcal{G}_{q,p}$ and satisfying condition (\ref{eqq 2}). Let  $\rho : \: (0,\, \infty) \:  \to \: (0, \,\infty) \:$ satisfying condition (\ref{eq 1}).  
 Put $\frac{q}{\bar{q}} = \frac{p}{\bar{p}}<1$.
Assume that for any $r > 0$, the following inequalities hold:
\begin{equation}\label{eq 8}
\displaystyle\int_{0}^{r} \frac{\rho(t)}{t} dt \lsim  \phi(r)^{\frac{q}{\bar{q}}-1}
\end{equation}
and 
\begin{equation}\label{eq 9}
\displaystyle\int_{r}^{\infty} \frac{\rho(t)\phi(t)}{t} dt \lsim  \phi(r)^{\frac{q}{\bar{q}}}.
\end{equation}
 Then, for all  $f \in (L^{q},L^{p})^{\phi}(\mathbb{R}^{d})$ we have
 $$\left\Vert T_{\rho}f \right\Vert_{(L^{\bar{q}},L^{\bar{p}})^{\psi}(\mathbb{R}^{d})} \lsim  \left\Vert f  \right\Vert_{(L^{q},L^{p})^{\phi}(\mathbb{R}^{d})},$$ where  $\psi(r)= (\phi(r))^{\frac{q}{\bar{q}}}$.
\end{theorem}
\begin{proof}
Suppose $1 < q\leq p < \infty$. Let $\phi$ be a surjective function belonging to $ \mathcal{G}_{q,p}$ and satisfying condition (\ref{eqq 2}). Let $\rho : \: (0,\, \infty) \:  \to \: (0,\, \infty) \:$ satisfying condition (\ref{eq 1}).  Let $f \in (L^{q},L^{p})^{\phi}(\mathbb{R}^{d})\setminus \{ 0 \}$ since the case $f=0$ is obvious. Let $r>0$ and $x \in \mathbb{R}^{d}$. We have
\begin{eqnarray*}
 T_{\rho} f(x)&=&\displaystyle\int_{B(x,\, r)}\frac{\rho(\left|x-y\right|)}{\left|x-y\right|^{d}}f(y)dy+\displaystyle\int_{\mathbb{R}^{d}\setminus B(x,\, r)}\frac{\rho(\left|x-y\right|)}{\left|x-y\right|^{d}}f(y)dy\\ &=& I_{1}(x) + I_{2}(x) . 
\end{eqnarray*}
Following the proof of Theorem B in \cite{GE0}, we have
\begin{eqnarray}  \label{rel 15}
 \vert I_{1}(x) \vert & \leq &  C M f(x) \phi(r)^{\frac{q}{\bar{q}}-1}.
\end{eqnarray}
Furthermore, we have
\begin{eqnarray*}
& &\vert I_{2}(x) \vert  \\
& \leq & C\sum\limits^{\infty}_{i=0}\frac{\rho(2^{i+1}r)}{(2^{i}r)^{d}} \vert B(x,2^{i+1}r) \vert^{\frac{d}{q^{'}}} \left( \displaystyle\int_{B(x,2^{i+1}r)} \vert f(y)\vert^{q} dy\right)^{\frac{1}{q}}\\
 &\leq & C \sum\limits^{\infty}_{i=0}\frac{\rho(2^{i+1}r)}{(2^{i}r)^{\frac{d}{q}}} (2^{i+1}r)^{\frac{d}{q}} \phi(2^{i+1}r)\frac{1}{\phi(2^{i+1}r)}(2^{i+1}r)^{-\frac{d}{q}} \Vert f \chi_{B(x,2^{i+1}r)}   \Vert_{q} \\
 &\leq & C \Vert f \Vert_{(L^{q},L^{\infty})^{\phi}(\mathbb{R}^{d})}\sum\limits^{\infty}_{i=0} \rho(2^{i+1}r) \phi(2^{i+1}r)\\&\leq & C \Vert f \Vert_{(L^{q},L^{\infty})^{\phi}(\mathbb{R}^{d})}\sum\limits^{\infty}_{i=0} \displaystyle\int_{2^{i}r}^{2^{i+1}r} \frac{\rho(t)\phi(t)}{t} dt \\&\leq & C \Vert f \Vert_{(L^{q},L^{\infty})^{\phi}(\mathbb{R}^{d})}\displaystyle\int_{r}^{+\infty} \frac{\rho(t)\phi(t)}{t} dt.
\end{eqnarray*}
Hence,
\begin{eqnarray}\label{rel 17}
\vert I_{2}(x) \vert \leq C \Vert f \Vert_{(L^{q},L^{\infty})^{\phi}(\mathbb{R}^{d})}\phi(r)^{\frac{q}{\bar{q}}}.
\end{eqnarray}
It follows from \eqref{rel 15} and \eqref{rel 17} that
\begin{eqnarray*}
\vert T_{\rho} f(x) \vert &\leq& C\left(M f(x) \phi(r)^{\frac{q}{\bar{q}}-1}+\Vert f \Vert_{(L^{q},L^{\infty})^{\phi}(\mathbb{R}^{d})}\phi(r)^{\frac{q}{\bar{q}}}\right)
\end{eqnarray*}
for all $ r > 0 $. \\ 
Since $ \phi $ is surjective, we can choose  $ r > 0 $ such that $$ \phi(r)= \frac{M f(x)}{\Vert f \Vert_{(L^{q},L^{\infty})^{\phi}(\mathbb{R}^{d})}}. $$
Thus, for every $ x \in \mathbb{R}^{d} $ we have  
\begin{eqnarray*} \label{rel 55}
\vert T_{\rho} f(x) \vert \leq C (M f(x))^{\frac{q}{\bar{q}}}\Vert f \Vert^{1-\frac{q}{\bar{q}}}_{(L^{q},L^{\infty})^{\phi}(\mathbb{R}^{d})}.
\end{eqnarray*}
It follows that 
\begin{eqnarray*}
_{r}\Vert T_{\rho} f \Vert_{\bar{q},\bar{p}} &\leq& C\Vert f \Vert^{1-\frac{q}{\bar{q}}}_{(L^{q},L^{\infty})^{\phi}(\mathbb{R}^{d})}\:_{r}\Vert M f \Vert^{\frac{q}{\bar{q}}}_{\bar{q}\frac{q}{\bar{q}},\bar{p}\frac{p}{\bar{p}}}\\
&\leq& C\Vert f \Vert^{1-\frac{q}{\bar{q}}}_{(L^{q},L^{\infty})^{\phi}(\mathbb{R}^{d})}\:_{r}\Vert M f \Vert^{\frac{q}{\bar{q}}}_{q,p}.
\end{eqnarray*}
Multiplying both sides of the previous inequality by $ \frac{1}{\psi(r)} r^{d(-\frac{1}{\bar{q}}-\frac{1}{\bar{p}})} $, where  $\psi(r)= (\phi(r))^{\frac{q}{\bar{q}}}$, we get 
\begin{eqnarray*}
&&\frac{1}{\psi(r)} r^{d(-\frac{1}{\bar{q}}-\frac{1}{\bar{p}})} \:_{r}\Vert T_{\rho} f \Vert_{\bar{q},\bar{p}} \\  &\leq& C\Vert f \Vert^{1-\frac{q}{\bar{q}}}_{(L^{q},L^{\infty})^{\phi}(\mathbb{R}^{d})}   \left(  \frac{1}{\phi(r)} r^{d(-\frac{1}{q}-\frac{1}{p})}\:_{r}\Vert M f \Vert_{q,p}  \right)^{\frac{q}{\bar{q}}}.
\end{eqnarray*}
Hence,  $$\left\Vert T_{\rho}f \right\Vert_{(L^{\bar{q}},L^{\bar{p}})^{\psi}(\mathbb{R}^{d})} \leq C\Vert f \Vert^{1-\frac{q}{\bar{q}}}_{(L^{q},L^{\infty})^{\phi}(\mathbb{R}^{d})} \left\Vert Mf  \right\Vert^{\frac{q}{\bar{q}}}_{(L^{q},L^{p})^{\phi}(\mathbb{R}^{d})}$$
By Theorem \ref{theo 1}, we get 
\begin{eqnarray*}\label{rel 18}
\left\Vert T_{\rho}f \right\Vert_{(L^{\bar{q}},L^{\bar{p}})^{\psi}(\mathbb{R}^{d})} \leq C\Vert f \Vert^{1-\frac{q}{\bar{q}}}_{(L^{q},L^{\infty})^{\phi}(\mathbb{R}^{d})}  \left\Vert f  \right\Vert^{\frac{q}{\bar{q}}}_{(L^{q},L^{p})^{\phi}(\mathbb{R}^{d})}.
\end{eqnarray*}
We end the proof by applying Proposition \ref{prop 1}.
\end{proof}
The analogue of Theorem \ref{theo 4} within the framework of generalized Morrey spaces was proved in \cite[Theorem B]{GE0}. 
\section{Olsen-type inequalities in generalized Fofana spaces}
\label{S5}
As a consequence of Proposition \ref{prop 2} and Theorem \ref{theo 2}, we have the following Olsen-type inequality.
\begin{theorem}\label{theo 5}
Assume that $1 < q\leq p \leq \infty$, $ 0 < \gamma < \frac{d}{q}$ and $-\frac{d}{q} \leq \beta  < -\gamma $. Let $\phi\in \mathcal{G}_{q,p}$ satisfying condition (\ref{eqq 2}) and $\phi(r) \lsim  r^{\beta}$ for every $r > 0$. Put $s=-\frac{\beta q}{\gamma}$ and $l=-\frac{\beta p}{\gamma}$. If $W \in (L^{s},L^{l})^{\phi^{\frac{q}{s}}}(\mathbb{R}^{d})$, 
 then for all  $f \in (L^{q},L^{p})^{\phi}(\mathbb{R}^{d})$ we have
 $$\left\Vert W \cdot I_{\gamma}f \right\Vert_{(L^{q},L^{p})^{\phi}(\mathbb{R}^{d})} \lsim \left\Vert W  \right\Vert_{(L^{s},L^{l})^{\phi^{\frac{q}{s}}}(\mathbb{R}^{d})} \left\Vert f  \right\Vert_{(L^{q},L^{p})^{\phi}(\mathbb{R}^{d})}.$$ 
\end{theorem}
\begin{remark}
Note that Theorem \ref{theo 5} in the case $p=\infty$ was proved in \cite[Theorem 3.2]{GE}. 
\end{remark}
\begin{proof}[Proof of Theorem \ref{theo 5}]
 Let $1 < q\leq p < \infty$. Let $\phi\in \mathcal{G}_{q,p}$ satisfying condition (\ref{eqq 2}). Put $\frac{1}{\bar{p}} = \frac{1}{p} -\frac{1}{l} \: \text{ and } \: \frac{1}{\bar{q}} = \frac{1}{q} -\frac{1}{s}.$
This implies that $$ \frac{1}{\frac{l}{p}} +\frac{1}{\frac{\bar{p}}{p}} =  1 \: \text{ and } \: \frac{1}{\frac{s}{q}} +\frac{1}{\frac{\bar{q}}{q}} =  1.$$
According to Proposition \ref{prop 2}, 
 $$\left\Vert W \cdot I_{\gamma}f \right\Vert_{(L^{q},L^{p})^{\phi}(\mathbb{R}^{d})} \leq \left\Vert W  \right\Vert_{(L^{s},L^{l})^{\phi^{\frac{q}{s}}}(\mathbb{R}^{d})} \left\Vert I_{\gamma}f  \right\Vert_{(L^{\bar{q}},L^{\bar{p}})^{\phi^{\frac{q}{\bar{q}}}}(\mathbb{R}^{d})}.$$ 
We end the proof by applying Theorem \ref{theo 2}.
\end{proof}
The below estimate is also an Olsen-type inequality involving the generalized fractional integral operator $T_{\rho}$. 

\begin{theorem}\label{theo 6}
Suppose that $1 < q\leq p < \infty$. Let $\phi$ be a surjective function belonging to $ \mathcal{G}_{q,p}$ and satisfying condition (\ref{eqq 2}). Let  $\rho : \: (0,\, \infty) \:  \to \: (0, \,\infty) \:$ satisfying condition (\ref{eq 1}) and the inequalities \eqref{eq 8} and \eqref{eq 9}.  
Put
 $$\frac{q}{\bar{q}} = \frac{p}{\bar{p}}<1,\, \frac{1}{l} = \frac{1}{p} -\frac{1}{\bar{p}} \: \text{ and } \: \frac{1}{s} = \frac{1}{q} -\frac{1}{\bar{q}}. $$
 If $W \in (L^{s},L^{l})^{\phi^{\frac{q}{s}}}(\mathbb{R}^{d})$, then for all  $f \in (L^{q},L^{p})^{\phi}(\mathbb{R}^{d})$ we have
 $$\left\Vert W \cdot T_{\rho}f \right\Vert_{(L^{q},L^{p})^{\phi}(\mathbb{R}^{d})} \lsim \left\Vert W  \right\Vert_{(L^{s},L^{l})^{\phi^{\frac{q}{s}}}(\mathbb{R}^{d})} \left\Vert f  \right\Vert_{(L^{q},L^{p})^{\phi}(\mathbb{R}^{d})}.$$  
\end{theorem}
Theorem \ref{theo 6} is a consequence of Proposition \ref{prop 2} and Theorem \ref{theo 4}. Its proof is similar to that of Theorem \ref{theo 5}.


\begin{thebibliography}{99}

\bibitem{DFS} M. Dosso, I. Fofana, M. Sanogo, \textit{On some subspaces of Morrey-Sobolev spaces and boundedness
of Riesz integrals}, Ann. Pol. Math., 108 (2013), 133-153.

\bibitem{Er} H. G. Eridani, \textit{On the boundedness of a generalized fractional integral on generalized Morrey
spaces}, Tamkang J. Math., \textbf{33}(4), (2002), 335-340.


\bibitem{JF} J. Feuto, \textit{Norm inequalities in some subspaces of Morrey space}, Ann. Math. Blaise Pascal, \textbf{21} (2) (2014), 21-37.


\bibitem{IF1} I. Fofana, \textit{Etude d'une classe d'espaces de fonctions contenant les espaces de Lorentz}, Afr. Mat., \textbf{2} (1) (1988), 29-50.
\bibitem{fof} I. Fofana, \textit{Continuit\'{e} de l'int\'{e}grale fractionnaire et espace $(L^{q},\, l^{ p})^{\alpha}$}, C. R. Acad. Sci. Paris, \textbf{308} (1) (1989), 525-527.

\bibitem{Fof} I. Fofana, \textit{Espace $(L^{q},\, l^{ p})^{\alpha}$ et continuit\'e de l'opérateur maximal fractionnaire de Hardy–Littlewood}, Afr. Mat. \textbf{3} (12) (2001), 23-37.

\bibitem{Fof-F-K} I. Fofana, F.R. Fal\'ea, B.A. Kpata, \textit{A class of subspaces of Morrey spaces and norm inequalities on Riesz potential operators}, Afr. Mat., \textbf{26} (5-6) (2015), 717-739.

 \bibitem{FS2} J.J.F. Fournier, J. Stewart, \textit{Amalgams of $L^{p}$ and $l^{q}$}, Bull. Amer. Math. Soc., \text{13} (1985), 1-21.
 \bibitem{GF} J. Garci\'a-Cuerva et J. R. de Francia, \textit{Weighted norm inequalities and related topics}, vol. 116, North-Holland Math. Stud., 1985.

 \bibitem{GoM} A. Gogatishvili, R. Mustafayev, \textit{Equivalence of norms
  of Riesz potential and fractional maximal function in Morrey-type
  spaces}, Preprint, Institute of Mathematics, AS CR, Prague. (2008), 7-14.


\bibitem{GS} V.S. Guliyev and P.S. Shukurov, \textit{On the Boundedness of the Fractional Maximal Operator, Riesz Potential and Their Commutators in Generalized Morrey Spaces.} 
In: Almeida, A., Castro, L., Speck, FO. (eds) Advances in Harmonic Analysis and Operator Theory. Operator Theory: Advances and Applications, vol 229. Birkhäuser, Basel.

 \bibitem{GE0} H. Gunawan, \textit{A note on the generalized fractional integral operators}, J. Indones. Math. Soc. (MIHMI), \textbf{9}(2003), 39-43.


\bibitem{GE} H. Gunawan and Eridani, \textit{Fractional integrals and generalized Olsen inequalities}, Kyungpook Math. J.,
      \textbf{49}(2009), 31-39.

 \bibitem{FH} F. Holland, \textit{Harmonic analysis on amalgams of $L^{p}$ and $l^{q}$}, J. London. Math. Soc., \textbf{(2)} 10, (1975), 295-305.


\bibitem{KNS} K. Kurata, S. Nishigaki, and S. Sugano, \textit{Boundedness of integral operators on generalized
Morrey spaces and its application to Schr\"odinger operators}, Proc. Amer. Math.
Soc.,\textbf{ 128}(2002), 1125-1134.

\bibitem{Ma} N. Maz'ya, \textit{Sobolev spaces}, Springer-Verlag, Berlin and New York, 1985.

\bibitem{Mi} T. Mizuhara, \textit{Boundedness of some classical operators on generalized Morrey spaces}, In: Harmonic Analysis, ICM 90 Satellite Proc. (Ed. S. Igari), Springer-Verlag, Tokyo (1991), 183-189.

\bibitem{NKD} P. Nagacy, B.A. Kpata, and N. Diarra, \textit{Boundedness of the Hardy-Littlewood maximal operator on
generalized Fofana spaces}, Preprint, arXiv:2605.21670v1 [math.FA] (2026). 


\bibitem{EN}E. Nakai, \textit{Hardy-Littlewood maximal operator, singular integral operators and the Riesz potentials on generalized Morrey spaces}, Math. Nachr., 166(1994), 95-103. 

\bibitem{EN1}E. Nakai, \textit{On generalized fractional integrals}, Taiwanese J. Math., 5(2001), 587-602.
             
\bibitem{Ol} P. A. Olsen, \textit{Fractional integration, Morrey spaces and a Schr\"odinger equation}, Comm. Partial Differential Equations, 20(1995), 2005-2055.

\bibitem{S-S-T} Y. Sawano, S. Sugano and H. Tanaka, \textit{A note on generalized fractional integral operators on generalized Morrey spaces}, 	Bound. Value Probl., Vol. 2009, Article ID 835865, 18 pages.


\bibitem{SST} Y. Sawano, S. Sugano and H. Tanaka, \textit{Olsen's inequality and its applications to Schr\"odinger equations}, RIMS Kôkyûroku Bessatsu \textbf{B26} (2011), 51-80.

\bibitem{Su} S. Sugano, \textit{Some inequalities for generalized fractional integral operators on generalized Morrey spaces and their remarks}, Sci. Math. Jpn., \textbf{76}(3) (2013), 471–495.


\bibitem{ST} S. Sugano and H. Tanaka, \textit{Boundedness of fractional integral operators on generalized
Morrey spaces}, Sci. Math. Jpn., \textbf{58} (2003), 531-540.

\bibitem{W} N. Wiener, \textit{On the representation of functions by trigonometrical integrals},  Math. Z. \textbf{24} (1926), 575-616.
\end{thebibliography}
\end{document}